\documentclass[11pt]{article}
\usepackage{amsmath,amsthm,amsfonts}
\usepackage{color}
\def\a{\mathbf{a}}
\def\bb{\mathbf{b}}
\def\m{\mathbf{m}}
\def\n{\mathbf{n}}
\def\C{\mathbb{C}}
\def\Z{\mathbb{Z}}
\def\Be{\mathcal{B}}
\newtheorem{theorem}{\hspace*{\parindent}Theorem}
\newtheorem{lemma}{\hspace*{\parindent}Lemma}

\DeclareMathOperator*{\res}{\mathrm{res}}
\title{A new identity for a sum of products of the generalized hypergeometric functions}
\author{Dmitrii Karp$^{\rm a}$\footnote{Corresponding author. E-mail: D.B.\:Karp -- \emph{dmitriibkarp@tdtu.edu.vn}, Alexey Kuznetsov --  \emph{kuznetsov@mathstat.yorku.ca}}~~and Alexey Kuznetsov$^{\rm b}$
\\[10pt]
\small{\textit{$\phantom{1}^a$Faculty of Mathematics and Statistics, Ton Duc Thang University, Ho Chi Minh City, Vietnam}}\\\small{\textit{$\phantom{1}^b$Department of Mathematics and Statistics, York University, Toronto, ON, Canada}}}
\date{}
\begin{document}
\maketitle

\begin{abstract}
Reduction formulas for sums of products of hypergeometric functions can be traced back to Euler. This topic has an intimate connection to summation and transformation formulas, contiguous relations and algebraic properties of the (generalized) hypergeometric differential equation. Over recent several years, important discoveries have been made in this subject by Gorelov, Ebisu, Beukers and Jouhet and Feng, Kuznetsov and Yang. In this paper, we give a generalization of Feng, Kuznetsov and Yang identity covering also Ebisu's and Gorelov's formulas as particular cases.
\end{abstract}

\bigskip

Keywords: \emph{Generalized hypergeometric function, duality relation, sum of products identity}

\bigskip

MSC2010: 33C20

\bigskip

\section{Introduction.} The history of transformation and reduction formulas for linear combinations of products of hypergeometric functions can be traced back to Euler, as his celebrated identity
$$
{}_2F_1\!\left(\!\begin{array}{c}a, b\\c\end{array}\!\!\vline\,\,x\right)=(1-x)^{c-a-b}{}_2F_1\!\left(\!\begin{array}{c}c-a,c-b\\c\end{array}\!\!\vline\,\,x\right)
$$
immediately leads to the reduction formula
\begin{equation*}
{}_2F_1\!\left(\!\begin{array}{c}a, b\\c\end{array}\!\!\vline\,\,x\right){}_2F_1\!\left(\!\begin{array}{c}1-a, 1-b\\2-c\end{array}\!\!\vline\,\,x\right)
-{}_2F_1\!\left(\!\begin{array}{c}c-a, c-b\\c\end{array}\!\!\vline\,\,x\right){}_2F_1\!\left(\!\begin{array}{c}1+a-c, 1+b-c\\2-c\end{array}\!\!\vline\,\,x\right)=0.
\end{equation*}
The symbol ${}_pF_{q}$ in the above formulas and throughout the paper stands for the generalized hypergeometric functions \cite[(2.1.2)]{AAR}. Later on, Gauss found a similarly flavored relation \cite[(6)]{Nesterenko1995}
\begin{multline*}
{}_2F_1\!\left(\!\begin{array}{c}a, b\\c\end{array}\!\!\vline\,\,x\right){}_2F_1\!\left(\!\begin{array}{c}-a, -b\\-c\end{array}\!\!\vline\,\,x\right)
\\
-x^2\frac{ab(c-a)(c-b)}{c^2(c^2-1)}{}_2F_1\!\left(\!\begin{array}{c}1-a,1-b\\2-c\end{array}\!\!\vline\,\,x\right){}_2F_1\!\left(\!\begin{array}{c}1+a,1+b, \\2+c\end{array}\!\!\vline\,\,x\right)=1
\end{multline*}
and Legendre discovered an identity for a linear combination of three products of the complete elliptic integrals $K$ and $E$ with complimentary arguments $k$ and $k'=\sqrt{1-k^2}$ (both $K$ and $E$ are a particular case of  the Gauss function ${}_2F_1$) \cite[Theorem~3.2.7]{AAR}.  Legendre's formula was generalized by Elliott in 1904 to a linear combination of three products of the Gauss hypergeometric functions containing three independent parameters \cite[Theorem~3.2.8]{AAR}. Another two-parameter generalization of Legendre's identity was found in \cite{AQVV} which was soon thereafter generalized in \cite{BNPV} to an identity with four independent parameters covering both Elliott's formula and the formula from \cite{AQVV}.  Another identity for a linear combination of three products of the Gauss functions with coefficients depending quadratically on the argument and containing three independent parameters was given in \cite[Theorem~3.2]{GKLZ}.

An independent development started in 1931 with the paper \cite{Darling} by Darling, where the author discovered probably the first reduction formulas for linear combinations of products of Clausen's  hypergeometric function ${}_3F_2$ (also called  Thomae's hypergeometric function by some authors). Soon thereafter Bailey \cite{Bailey} found a new method for proving Darling's identity which also lead to its generalizations.   It seems that the first author to obtain a reduction formula for a linear combination of products of the generalized hypergeometric functions ${}_{p+1}F_{p}$ with $p>3$ was Nesterenko in his work on Hermite-Pad\'{e}  approximation \cite[Theorem~5]{Nesterenko1995}. Probably due to complicated notation his work remained largely unnoticed. It was referred to, however, in Gorelov's paper \cite{Gorelov2010}, where the author discovered another extension of Darling's formulas to general ${}_pF_{q}$ \cite[Corollary~1]{Gorelov2010}.

The years 2015-2016 marked an important breakthrough in the subject. Beukers and Jouhet published the paper \cite{BeukersJouhet} in 2015, where they presented a very general identity involving derivatives up to certain order of the generalized hypergeometric function ${}_{p+1}F_{p}$ and gave a purely algebraic proof based on the theory of general differential and difference modules.  Soon thereafter the second author jointly with Feng and Yang found in \cite{FKY} another identity for the sum of products of general ${}_pF_{q}$ functions, involving more free parameters than Beukers-Jouhet formula and coinciding with it if the order of derivatives is set to zero. This particular case of Beukers-Jouhet's formula also coincides with Gorelov's formula \cite[(4)]{Gorelov2010} (for $p<q+1$ one also has to apply confluence to derive Gorelov's formula from that of  Beukers-Jouhet). The proof in \cite{FKY} is entirely different from that of the previous authors and is based on the so-called non-local derangement identity.  Finally, we mention a very important work of Ebisu \cite{Ebisu}, who already in 2012 discovered an identity for the sum of products of the Gauss hypergeometric function ${}_2F_1$ which after certain change  can be viewed as a generalization of the particular ${}_2F_1$ case of the Feng, Kuznetsov and Yang formula \cite[Theorem~1]{FKY}.  In \cite{KKSiberian2018} the first author (jointly with S.I.\:Kalmykov) gave an independent derivation of a generalization of the particular ${}_1F_1$ case of the Feng, Kuznetsov and Yang formula, which can also be obtained by applying confluence to Ebisu's result.  Two independent linearization identities for the sum of products of the Gauss functions were obtained by the first author in \cite[(9),(14)]{KarpJMS2016}.

The purpose of this note to give a generalization of both Feng, Kuznetsov and Yang  identity \cite[Theorem~1]{FKY} and Ebisu's identity   \cite[Theorem~3.7]{Ebisu}. We do so by introducing an additional vector of integer shifts in Feng, Kuznetsov and Yang  identity.  Our formula reduces to that of Ebisu in the case of the Gauss hypergeometric function ${}_2F_{1}$.  The proof is a generalization of the proof in  \cite{FKY} and hinges on the interpretation of the coefficients of the sum of products of the hypergeometric functions as a sum of residues of a specially tailored  rational function.

\section{Main results}

Suppose $r\ge2$ is an integer, $\a=(a_1,\ldots,a_{r})\in\C^{r}$, $\bb=(b_1,\ldots,b_{r})\in\C^{r}$,
$\m=(m_1,\ldots,m_r)\in\Z^r$  and  $\n=(n_1,\ldots,n_r)\in\Z^r$.  Denote further
$$
M=m_1+\cdots+{m_r},~~~N=n_1+\cdots+{n_r}.
$$
We use the standard notation $\Gamma(z)$ for Euler's gamma function and $(z)_k=\Gamma(z+k)/\Gamma(z)$, $k\in\Z$, for the rising factorial (note that $k$ can be any integer). It is convenient to introduce the shorthand notation for the products
$$
\Gamma(\a+z)=\Gamma(a_1+z)\cdots\Gamma(a_r+z),~~~(\a)_{\n}=(a_1)_{n_1}(a_2)_{n_2}\cdots(a_r)_{n_r}.
$$
The symbol $\a_{[j]}$ stands for the vector $\a$ with $j$-th component omitted:
$$
\a_{[j]}=(a_1,\ldots,a_{j-1},a_{j+1},\ldots,a_r),
$$
and let $\a+\beta=(a_1+\beta,\ldots,a_r+\beta)$. In particular,
$$
(a_i-\a_{[i]})=(a_i-\a_{[i]})_{1}=(a_i-a_1)(a_i-a_2)\cdots(a_{i}-a_{i-1})(a_{i}-a_{i+1})\cdots(a_{i}-a_r).
$$
Define
\begin{equation}\label{eq:mminnmax}
m_{\min}=\min\limits_{1\le{i}\le{r}}(m_i),~~~~~n_{\max}=\max\limits_{1\le{i}\le{r}}(n_i),
\end{equation}
and
\begin{equation}\label{eq:p-defined}
p=\max\{-1,M-N-r+1\}.
\end{equation}
Our main result is the following theorem.
\begin{theorem}\label{th:main}
 Assume that the components of $\a$ are distinct modulo integers. Then for some complex numbers $\beta_j$
\begin{multline}\label{eq:main}
\sum\limits_{i=1}^{r}\frac{(1-\bb+a_i)_{\m-n_i}z^{-n_i}}{(a_i-\a_{[i]})_{\n_{[i]}-n_i+1}}
{}_{r}F_{r-1}\bigg(\begin{matrix}\bb-a_i\\1+\a_{[i]}-a_i\end{matrix}\Big\vert z\bigg)
{}_{r}F_{r-1}\bigg(\begin{matrix}1-\bb+a_i+\m-n_i\\1-\a_{[i]}+a_i+\n_{[i]}-n_i\end{matrix}\Big\vert z\bigg)
\\
=(1-z)^{-p-1}\sum\limits_{j=-n_{\max}}^{p-m_{\min}}\beta_jz^{j}.
\end{multline}\end{theorem}

\textit{Proof of Theorem~\ref{th:main}}.  Using the Cauchy product expand:
\begin{multline}\label{eq:Th1proofchain}
S(z):=\sum\limits_{i=1}^{r}\frac{(1-\bb+a_i)_{\m-n_i}z^{-n_i}}{(a_i-\a_{[i]})_{\n_{[i]}-n_i+1}}
{}_{r}F_{r-1}\bigg(\begin{matrix}\bb-a_i\\1+\a_{[i]}-a_i\end{matrix}\Big\vert z\bigg)
{}_{r}F_{r-1}\bigg(\begin{matrix}1-\bb+a_i+\m-n_i\\1-\a_{[i]}+a_i+\n_{[i]}-n_i\end{matrix}\Big\vert z\bigg)
\\
=\sum\limits_{i=1}^{r}\frac{(1-\bb+a_i)_{\m-n_i}z^{-n_i}}{(a_i-\a_{[i]})_{\n_{[i]}-n_i+1}}
\sum\limits_{k=0}^{\infty}z^k\sum\limits_{j=0}^{k}
\frac{(\bb-a_i)_j(1-\bb+a_i+\m-n_i)_{k-j}}{(1+\a_{[i]}-a_i)_jj!(1-\a_{[i]}+a_i+\n_{[i]}-n_i)_{k-j}(k-j)!}
\\
=\!\sum\limits_{i=1}^{r}\sum\limits_{k=0}^{\infty}z^{k-n_i}\!\!\sum\limits_{j=0}^{k}
\!\frac{(1-\bb+a_i)_{\m-n_i}(\bb-a_i)_j(1-\bb+a_i+\m-n_i)_{k-j}}{(a_i-\a_{[i]})_{\n_{[i]}-n_i+1}(1+\a_{[i]}-a_i)_jj!(1-\a_{[i]}+a_i+\n_{[i]}-n_i)_{k-j}(k-j)!}
\\
=\sum\limits_{i=1}^{r}\sum\limits_{k=0}^{\infty}z^{k-n_i}\sum\limits_{j=0}^{k}\gamma^{k}_{i,j}
=\sum\limits_{i=1}^{r}\sum\limits_{k_i=-n_i}^{\infty}z^{k_i}\sum\limits_{j=0}^{k_i+n_i}\gamma^{k_i+n_i}_{i,j},
\end{multline}
where
$$
\gamma^{k}_{i,j}=\frac{(1-\bb+a_i)_{\m-n_i}(\bb-a_i)_j(1-\bb+a_i+\m-n_i)_{k-j}}
{\left(a_i-\a_{[i]}\right)_{\n_{[i]}-n_i+1}\left(1+\a_{[i]}-a_i\right)_jj!\left(1-\a_{[i]}+a_i+\n_{[i]}-n_i\right)_{k-j}(k-j)!}.
$$
Straightforward calculation using $(z)_k=\Gamma(z+k)/\Gamma(z)$ and $(z)_j=(-1)^j(1-z-j)_{j}$  yields
$$
(1-\bb+a_i)_{\m-n_i}(\bb-a_i)_j(1-\bb+a_i+\m-n_i)_{k-j}=(-1)^{jr}(1-\bb+a_i-j)_{\m+k-n_i}
$$
Similarly,
$$
\left(a_i-\a_{[i]}\right)_{\n_{[i]}-n_i+1}\left(1+\a_{[i]}-a_i\right)_j\left(1-\a_{[i]}+a_i+\n_{[i]}-n_i\right)_{k-j}
=(-1)^{j(r-1)}(-\a_{[i]}+a_i-j)_{\n_{[i]}+k-n_i+1},
$$
so that
$$
\gamma^{k}_{i,j}=\frac{(-1)^{j}(1-\bb+a_i-j)_{\m+k-n_i}}{(-\a_{[i]}+a_i-j)_{\n_{[i]}+k-n_i+1}j!(k-j)!}
$$
and
$$
\gamma^{k+n_i}_{i,j}=\frac{(-1)^{j}(1-\bb+a_i-j)_{\m+k}}{(-\a_{[i]}+a_i-j)_{\n_{[i]}+k+1}j!(k+n_i-j)!}.
$$
Note that by interpreting $j!=\Gamma(j+1)$ and $(k+n_i-j)!=\Gamma(k+n_i-j+1)$ we get  $\gamma^{k+n_i}_{i,j}=0$ for $j<0$ and $k+n_i<j$ which extends the above definition of $\gamma^{k+n_i}_{i,j}$ to these values of the indices. In view of this convention, we will have by rearranging the last expression in \eqref{eq:Th1proofchain}
$$
S=\sum\limits_{k=-n_{\max}}^{\infty}z^{k}\sum\limits_{i=1}^{r}\sum\limits_{j=0}^{k+n_i}\gamma^{k+n_i}_{i,j},
$$
Next note that $k+m_i\ge0$ for all $i=1,\ldots,r$ if $k\ge-m_{\min}$.  If $-m_{\min}\le-n_{\max}$, then $k\ge-m_{\min}$ for all terms in the above sum.  Otherwise, if $-m_{\min}>-n_{\max}$ we can write
$$
S(z)=\sum\limits_{k=-n_{\max}}^{-m_{\min}-1}\alpha_kz^{k}
+\sum\limits_{k=-m_{\min}}^{\infty}z^{k}\sum\limits_{i=1}^{r}\sum\limits_{j=0}^{k+n_i}\gamma^{k+n_i}_{i,j}
=\sum\limits_{k=-n_{\max}}^{-m_{\min}-1}\alpha_kz^{k}+S_1(z),
$$
where
$$
\alpha_k=\sum\limits_{i=1}^{r}\sum\limits_{j=0}^{k+n_i}\gamma^{k+n_i}_{i,j}.
$$
The purpose of the subsequent argument is to prove that $S_1(z)$ can be written as $z^{-m_{\min}}P_p(z)/(1-z)^{p+1}$, where
$$
p=\max\{-1,M-N-r+1\}
$$
and $P_p(z)$ is a polynomial of degree $p$ with the convention that $P_{-1}(z)\equiv0$. To this  end, for $k\in\Z$ introduce the sequence of rational functions:
\begin{equation}\label{eq:def_rational_f}
f_{k}(z)=\frac{(z-\bb-k+1)_{\m+k}}{(z-\a-k)_{\n+k+1}}.
\end{equation}
Note first that these functions are well defined and rational for both nonnegative and negative $k$.  Secondly,
for $k\ge-m_{\min}$ all rising factorials in the product $(z-\bb-k+1)_{\m+k}$ are polynomials (possibly of degree $0$) and all poles of the function $f_{k}(z)$ are points where $(z-a_i-k)_{k+n_i+1}=0$ and $k+n_i+1>0$, that is the points $z=a_i+k-j$ for $j=0,\ldots,k+n_i$.   Assuming that all poles of $f_k$ are simple for $k+n_i+1>0$ and $j=0,\ldots,k+n_i$ we will have by a straightforward calculation:
$$
\res\limits_{z=a_i+k-j}f_k(z)=\frac{(-1)^j(1-\bb+a_i-j)_{\m+k}}{(-\a_{[i]}+a_i-j)_{\n_{[i]}+k+1}j!(k+n_i-j)!}=\gamma^{k+n_i}_{i,j}.
$$
Hence, for all $k\ge-m_{\min}$
$$
\sum\limits_{\text{over all poles of }f_k(z)}\res{f_k(z)}=\sum\limits_{i=1}^{r}\sum\limits_{j=0}^{k+n_i}\gamma^{k+n_i}_{i,j},
$$
where only the terms with $k+n_i\ge0$ are non-vanishing.

Next, we use the fact that the sum of residues of a rational function at all finite points equals its residue at infinity \cite[(4.1.14)]{AblFokas}, which is the coefficient at $z^{-1}$ in the
asymptotic expansion
\begin{equation}\label{eq:fk-asymp}
f_k(z)\sim\sum\limits_{j=-\infty}^{M-N-r}C_{j}(k)z^{j},~~~\text{as}~z\to\infty,
\end{equation}
so that
\begin{equation}\label{eq:residue}
\sum\limits_{i=1}^{r}\sum\limits_{j=0}^{k+n_i}\gamma^{k+n_i}_{i,j}=C_{-1}(k).
\end{equation}

Our key observation is the following lemma.
\begin{lemma}\label{lm:residue-polynomial}
Let $p$ be defined in \eqref{eq:p-defined}. Then $C_{-1}(k)$ is a polynomial in $k$ of degree $p$ if $p\ge0$ and $C_{-1}=0$ if $p=-1$.
\end{lemma}
\begin{proof}
It is easy to see from definition \eqref{eq:def_rational_f} of $f_k(z)$ that  $C_{-1}(k)=0$ if $p<0$ and $C_{-1}(k)=1$  if $p=0$.  Assume that $p>0$ and write in view of $(z-\bb-k+1)_{\m+k}=\Gamma(z-\bb+\m+1)/\Gamma(z-\bb-k+1)$:
$$
\log[f_k(z)]=\sum_{i=1}^{r}\Bigl\{\log\Gamma(z-b_i+m_i+1)-\log\Gamma(z-b_i+1-k)-\log\Gamma(z-a_i+n_i+1)+\log\Gamma(z-a_i-k)\Bigr\}.
$$
Next, we apply Hermite's asymptotic expansion for $\log\Gamma(z+a)$ \cite[(1.8)]{Nemes}
$$
\log\Gamma(z+a)\sim(z+a-1/2)\log{z}-z\frac{1}{2}+\log(2\pi)+\sum\limits_{j=2}^{\infty}\frac{(-1)^j\Be_{j}(a)}{j(j-1)z^{j-1}},
$$
as $|z|\to\infty$ in the domain $|\arg z|<\pi-\delta$,  $0<\delta<\pi$. Here $\Be_j(x)$ stands for the $j$-th Bernoulli polynomial \cite[24.2.3]{NIST} given in terms of Bernoulli numbers $\Be_{n}$ by
$$
\Be_n(x)=\sum_{l=0}^{n}\binom{n}{l}\Be_{n-l}x^l,~~~\Be_{0}=1,
$$
so that the leading coefficient of $\Be_n(x)$ is $1$. Substituting Hermite's expansion for each gamma function in the above formula for $\log[f_k(z)]$, we get after some cancelations:
$$
\log[f_k(z)]\sim \log(z^{p-1})+\sum\limits_{j=1}^{\infty}\frac{(-1)^{j+1}Q_{j}(k)}{(j)_2z^j},
$$
where $p-1=M-N-r\ge0$ by assumption and
$$
Q_{j}(k)=\sum\limits_{i=1}^{r}\left[\Be_{j+1}(-a_i-k)-\Be_{j+1}(-b_i+1-k)+\Be_{j+1}(1-b_i+m_i)-\Be_{j+1}(1-a_i+n_i)\right]
$$
is a polynomial in $k$ of degree $j$  (remember that the leading coefficient of $\Be_{j+1}(x)$ is $1$, so that the degree is decreased by $1$ due to cancelation of the leading terms). Exponentiating this expression, we find that
$$
z^{1-p}f_k(z)\sim 1+\sum\limits_{s\ge1}\frac{q_s(k)}{z^s},
$$
where, according to the well-known formula for an exponential of a power series  (see \cite[Lemma~1]{KPJMS2018} and references therein),
$$
q_s(k)=\sum\limits_{l=1}^{s}\frac{1}{l!}\sum\limits_{\substack{{s_1+\cdots+s_l=s}\\{s_i\ge1}}}\prod\limits_{i=1}^{l}Q_{s_i}(k).
$$
This expression shows that $q_s(k)$ is a polynomial in $k$ of degree $s$, implying that the coefficient $q_p(k)$ in front of $z^{-p}$ in the Maclaurin series of $z^{1-p}f_k(z)$ is a polynomial in $k$ of degree $p$.
\end{proof}

\textit{Returning to the proof of Theorem~\ref{th:main}} and writing $q_p(k)=\sum_{l=0}^{p}\nu_{l}k^l$ we
obtain, in view of \eqref{eq:residue} and Lemma~\ref{lm:residue-polynomial},
\begin{multline*}
S_1(z)=\sum\limits_{k=-m_{\min}}^{\infty}z^{k}\sum\limits_{i=1}^{r}\sum\limits_{j=0}^{k+n_i}\gamma^{k+n_i}_{i,j}
=\sum\limits_{k=-m_{\min}}^{\infty}z^{k}C_{-1}(k)
=\sum\limits_{k=-m_{\min}}^{\infty}z^{k}q_p(k)
\\
=z^{-m_{\min}}\sum\limits_{j=0}^{\infty}q_p(j-m_{\min})z^j
=z^{-m_{\min}}\sum_{l=0}^{p}\nu_{l}\sum\limits_{j=0}^{\infty}(j-m_{\min})^lz^j.
\end{multline*}
It is easy to check that
$$
\sum\limits_{j=0}^{\infty}(j-m_{\min})^lz^j=\frac{\tilde{P}_{l}(z)}{(1-z)^{l+1}},
$$
where $\tilde{P}_l(z)$ is a polynomial of degree $l$ except for $m_{\min}=-1$ when $\tilde{P}_l(z)$ is a polynomial of degree $l-1$. Consequently,
$$
S_1(z)=\sum\limits_{k=-m_{\min}}^{\infty}q_p(k)z^k=z^{-m_{\min}}\sum_{l=0}^{p}\nu_{l}\frac{\tilde{P}_{l}(z)}{(1-z)^{l+1}}
=\frac{z^{-m_{\min}}P_{p}(z)}{(1-z)^{p+1}},
$$
where in accordance to the previous convention $P_{-1}(z)=0$.  Hence, finally,
$$
S(z)=\sum\limits_{k=-n_{\max}}^{-m_{\min}-1}\alpha_kz^{k}+S_1(z)=\sum\limits_{k=-n_{\max}}^{-m_{\min}-1}\alpha_kz^{k}
+\frac{z^{-m_{\min}}P_{p}(z)}{(1-z)^{p+1}}=(1-z)^{-p-1}\sum\limits_{j=-n_{\max}}^{p-m_{\min}}\beta_jz^{j}.~~~~\square
$$

\medskip

Next, we extend formula \eqref{eq:main} to the hypergeometric functions ${}_sF_{r-1}$ with $0\le{s}\le{r-1}$, $r\ge2$.
Assume that $\bb=(b_1,\ldots,b_s)\in\C^s$, $\a=(a_1,\ldots,a_r)\in\C^{r}$, $\m=(m_1,\ldots,m_s)\in\Z^s$  and  $\n=(n_1,\ldots,n_r)\in\Z^r$.  In agreement with the previous notation, write
$$
M=m_1+\cdots+{m_s},~~~N=n_1+\cdots+{n_r};
$$
and denote by $[x]$ the largest integer not exceeding $x$;  the symbols $m_{\min}$ and $n_{\max}$ retain their (intuitive) meanings from above.

\begin{theorem}\label{th:confluent}
Suppose $0\le{s}\le{r-1}$, $r\ge2$, and the components of $\a$ are distinct modulo integers.  Define  $p=[(M-N-r+1)/(r-s)]$. Then, for some complex number $\beta_j$,
\begin{multline}\label{eq:confluent}
\sum\limits_{i=1}^{r}\frac{(1-\bb+a_i)_{\m-n_i}z^{-n_i}}{(-1)^{s-r}(a_i-\a_{[i]})_{\n_{[i]}-n_i+1}}
{}_{s}F_{r-1}\bigg(\begin{matrix}\bb-a_i\\1+\a_{[i]}-a_i\end{matrix}\Big\vert z\bigg)
{}_{s}F_{r-1}\bigg(\begin{matrix}1-\bb+a_i+\m-n_i\\1-\a_{[i]}+a_i+\n_{[i]}-n_i\end{matrix}\Big\vert (-1)^{r-s}z\bigg)
\\
=\sum\limits_{j=-n_{\max}}^{\max(-m_{\min}-1,p)}\beta_jz^{j}.
\end{multline}
\end{theorem}
\begin{proof}
Follow the proof of Theorem~\ref{th:main} up to formula \eqref{eq:fk-asymp} taking here the following form
$$
f_k(z)\sim\sum\limits_{j=-\infty}^{M-N-r-(r-s)k}C_{j}(k)z^{j},~~~\text{as}~z\to\infty.
$$
This implies that $C_{-1}(k)=0$ for $M-N-r-(r-s)k<-1$, i.e. for $(r-s)k>M-N-r+1$. Define $p=[(M-N-r+1)/(r-s)]$. Then
$$
S_1(z)=\sum\limits_{k=-m_{\min}}^{\infty}C_{-1}(k)z^{k}
=\left\{\!\!\!\begin{array}{ll}\sum_{k=-m_{\min}}^{(r-s)k\le{M-N-r+1}}C_{-1}(k)z^{k}, & p\le{m_{\min}}\\[7pt]0, & p>{m_{\min}}.\end{array}\right.
$$
Hence, the ultimate line of the proof of Theorem~\ref{th:main} under conditions of Theorem~\ref{th:confluent} becomes:
$$
S(z)=\sum\limits_{k=-n_{\max}}^{-m_{\min}-1}\alpha_kz^{k}+S_1(z)=\sum\limits_{j=-n_{\max}}^{\max(-m_{\min}-1,p)}\beta_jz^{j}.
$$
\end{proof}

\noindent\textbf{Remark.} For $r=2$, $s=0$ we can apply the relation \cite[section~10.2]{NIST}
$$
J_{\nu}(x)=\frac{(x/2)^{\nu}}{\Gamma(\nu+1)}{}_0F_{1}\bigg(\begin{matrix}-\\\nu+1\end{matrix}\Big\vert -\frac{x^2}{4}\bigg)
$$
for the Bessel function $J_{\nu}$ to rewrite this case of identity \eqref{eq:confluent} in terms of the Bessel functions.
After some rearrangement this yields ($m\in\Z$):
$$
(-1)^mJ_{-\nu}(x)J_{\nu+m}(x)-J_{\nu}(x)J_{-\nu-m}(x)=\sum\limits_{j=0}^{2j<|m|}\frac{\beta_j}{x^{|m|-2j}},
$$
where the coefficients $\beta_j$ can be computed explicitly using the fact that $J_{\nu}(x)J_{\mu}(x)$ can be written as ${}_2F_{3}$ \cite[7.2(49)]{Bateman}.  The left hand side of the above formula can be interpreted as a ''generalized Wronskian''. Certainly, formula \cite[7.2(49)]{Bateman} can be used to give a direct proof of the above relation. We believe that it is known, but we could not immediately locate it in the literature.

\bigskip
\bigskip
\noindent{\bf Acknowledgments.} We thank Professor Leonid Kovalev (Syracuse University, USA) for useful discussions regarding Lemma~\ref{lm:residue-polynomial}.


\end{document}